\documentclass[a4paper,12pt]{article}

\usepackage[a4paper, top=3cm, left=2.5cm, right=2.5cm, bottom=3cm]{geometry}

\usepackage{amsmath, amsthm, pb-diagram}
\usepackage{amssymb}
\usepackage{amsfonts}
\usepackage{enumerate, fancyhdr, dsfont}

    \newcommand{\thzfc}{\mathrm{ZFC}}

    \newcommand{\Awf}{\mathcal{A}}

    \newcommand{\Dwf}{\mathcal{D}}

    \newcommand{\Mwf}{\mathcal{M}}

    \newcommand{\Uwf}{\mathcal{U}}

    \newcommand{\afrak}{\mathfrak{a}}
    \newcommand{\bfrak}{\mathfrak{b}}
    \newcommand{\cfrak}{\mathfrak{c}}
    \newcommand{\dfrak}{\mathfrak{d}}

    \newcommand{\rfrak}{\mathfrak{r}}
    \newcommand{\sfrak}{\mathfrak{s}}

    \newcommand{\menos}{\smallsetminus}

    \newcommand{\frestr}{\!\!\upharpoonright\!\!}

    \newcommand{\cov}{\mbox{\rm cov}}
    \newcommand{\non}{\mbox{\rm non}}
    
    \newcommand{\limdir}{\mbox{\rm limdir}}

    \newcommand{\Cor}{\mathds{C}}
    \newcommand{\Dor}{\mathds{D}}
    \newcommand{\Eor}{\mathds{E}}

    \newcommand{\Por}{\mathds{P}}
    \newcommand{\Qor}{\mathds{Q}}
    \newcommand{\Ror}{\mathds{R}}
    \newcommand{\Sor}{\mathds{S}}
    
    \newcommand{\Qnm}{\dot{\mathds{Q}}}

    \newcommand{\cf}{\mbox{\rm cf}}

    \newcommand{\sii}{{\ \mbox{$\Leftrightarrow$} \ }}

\title{Preservation properties for iterations with finite support}
\author{Diego A. Mej\'ia\thanks{Supported by the Monbukagakusho (Ministry of Education, Culture, Sports, Science and Technology) Scholarship, Japan.}
}

\date{\small Graduate School of System Informatics\\ Kobe University\\ Kobe, Japan.\\ \ \\ \texttt{damejiag@kurt.scitec.kobe-u.ac.jp}}

\begin{document}

\makeatletter
\def\@roman#1{\romannumeral #1}
\makeatother

\theoremstyle{plain}
  \newtheorem{theorem}{Theorem}[section]
  \newtheorem{corollary}[theorem]{Corollary}
  \newtheorem{lemma}[theorem]{Lemma}
  \newtheorem{prop}[theorem]{Proposition}
  \newtheorem{clm}[theorem]{Claim}
  \newtheorem{exer}[theorem]{Exercise}
  \newtheorem{question}[theorem]{Question}
\theoremstyle{definition}
  \newtheorem{definition}[theorem]{Definition}
  \newtheorem{example}[theorem]{Example}
  \newtheorem{remark}[theorem]{Remark}
  \newtheorem{context}[theorem]{Context}
  \newtheorem*{acknowledgements}{Acknowledgements}

\maketitle

\begin{abstract}
   We present the classical theory of preservation of $\sqsubset$-unbounded families in generic extensions by ccc posets, where $\sqsubset$ is a definable relation of certain type on spaces of real numbers, typically associated with some classical cardinal invariant. We also prove that, under some conditions, these preservation properties can be preserved in direct limits of an iteration, so applications are extended beyond the context of finite support iterations. Also, we make a breve exposition of Shelah's theory of forcing with an ultrapower of a poset by a measurable cardinal.
\end{abstract}

\section{Introduction}\label{SecIntro}

In this paper, we discuss two topics of technical nature that can be applied to forcing iterations with finite support. The first is about preservation properties of $\sqsubset$-unbounded families in forcing extensions, where $\sqsubset$ is a definable relation on spaces of real numbers as explained in Context \ref{ContextPres}. These type of preservation properties are introduced in \cite{JuSh-KunenMillerchart} and \cite{Br-Cichon}, later generalized and summarized in \cite[Sect. 6.4 and 6.5]{BarJu} and \cite{Go-Tools}. As these properties can be preserved under finite support iterations (fsi) of posets with the countable chain condition (ccc), the main application has been in the context of cardinal invariants, where the preservation property is used to preserve some cardinal invariant small while, with the reals added through the iteration, some other cardinal invariant becomes larger in the final generic extension.

The second topic is about forcing with an ultrapower of a poset by a measurable cardinal $\kappa$, originally introduced by Shelah \cite{Sh-TempIt} to show that, given a ccc poset $\Por$, the ultrapower of $\Por$ destroys all the maximal almost disjoint (mad) families of size $\geq\kappa$ that exist in the $\Por$-extension (see Corollary \ref{UltrapowDestrMAD}). This was used by Shelah to produce a ccc forcing notion that forces $\kappa<\dfrak<\afrak$ where $\dfrak$ is the \emph{dominating number} (see Example \ref{ExBdRelation}(1)) and $\afrak$ is the least size of an infinite mad family. Also, Shelah modified the construction of the model to get the consistency of $\aleph_1<\dfrak<\afrak$ without the use of a measurable cardinal.

This paper does not contain original results by the author and only contains technical results. The main purpose is to explain the two topics mentioned above under the point of view of the author, this as a prelude of the main results in \cite{Me-TempIt}. Some known facts about these topics that are not proven (and not even explicitly stated) in any other article or book are presented in this article, for instance:

\begin{itemize}
  \item The preservation property of Definition \ref{DefGood} is preserved in direct limits under some conditions (Theorem \ref{PresGoodLimdir}). This allows to preserve this property under some iterations of stronger type than fsi, e.g., \emph{template iterations}\footnote{This iteration technique was created by Shelah \cite{Sh-TempIt}. See also \cite{Br-TempIt}, \cite{Br-Luminy} and \cite{Me-TempIt} for further discussions.} (\cite[Sect. 4]{Me-TempIt}).
  \item Preservation of $\sqsubset$-unbounded reals (see Definition \ref{DefPresUnbreal}) under parallel direct limits (Theorem \ref{PresUnbrealLimit}). This fact simplifies the proof of the author's result stating that, in a certain type of template iteration, a real added at some stage of the iteration cannot be added at any other stage \cite[Thm. 4.16 and 4.18]{Me-TempIt}.
  \item A characterization about forcing a projective statement of real numbers with the ultrapower of a ccc poset (Theorem \ref{ProjStatandUltrapow}). This fact implies directly Shelah's result discussed above about destroying mad families with ultrapowers.
\end{itemize}

This article is structured in three parts. In Section \ref{SecCorr}, we explain correctness and direct limits, elementary facts about forcing that are essential for the construction of iterations with finite supports (e.g. template iterations). Section \ref{SecPresProp} is devoted to the topic of preservation properties on iterations with finite supports and, in Section \ref{SecUltrapow}, we discuss forcing with ultrapowers.

\begin{acknowledgements}
   The author is deeply grateful with professor J. Brendle for all his help and guidance, in particular, with the topics of preservation properties and template iterations that the author learnt directly from him.

   The author is also thankful with professor S. Fuchino for his invitation to the RIMS 2013 conference.
\end{acknowledgements}

\section{Correctness and direct limits}\label{SecCorr}

The concept of correctness is originally developed for complete Boolean algebras \cite{Br-Amalg,Br-Shat,Br-Luminy}, but notions and results can be translated in terms of posets in general. In this section, we present correctness for posets.

Usually, if $\Por$ and $\Qor$ are posets, $\Por\lessdot\Qor$ denotes that $\Por$ is completely embedded into $\Qor$. For this article, we reserve the notation $\Por\lessdot\Qor$ to say that $\Por$ is a complete suborder of $\Qor$. Also, if $M$ is a transitive model of (a quite large finite fragment of) $\thzfc$ and $\Por\in M$, $\Por\lessdot_M\Qor$ denotes that $\Por\subseteq\Qor$ and that any maximal antichain of $\Por$ in $M$ is also a maximal antichain of $\Qor$.

For this section, fix $M\subseteq N$ transitive models of $\thzfc$. Note that, if $\Por\in M$ and $\Qor\in N$ are posets, $\Por\lessdot_M \Qor$ implies that, whenever $G$ is $\Qor$-generic over $N$, $\Por\cap G$ is $\Por$-generic over $M$ and $M[\Por\cap G]\subseteq N[G]$.

Recall the four-element lattice $I_4:=\{\wedge,0,1,\vee\}$ where $\vee$ is the largest element, $\wedge$ is the least element and $0,1$ are in between.

\begin{definition}[Correct system of embeddings]\label{DefCorr}
   Let $\Por_i$ be a poset for each $i\in I_4$ and assume that $\Por_i\lessdot\Por_j$ for $i<j$ in $I_4$. We say that the system $\langle\Por_\wedge,\Por_0,\Por_1,\Por_\vee\rangle$ is \emph{correct} if, for each $p\in\Por_0$ and $q\in\Por_1$, if both have compatible reductions in $\Por_\wedge$, then $p$ and $q$ are compatible in $\Por_\vee$. An equivalent statement is that, for each $p\in\Por_0$ and for every reduction $r\in\Por_\wedge$ of $p$, $r$ is a reduction of $p$ with respect to $\Por_1,\Por_\vee$.

   There is a restrictive version of this notion. For the model $M$, if $\Por_\wedge,\Por_0\in M$, $\Por_\wedge\lessdot\Por_0$, $\Por_\wedge\lessdot_M\Por_1$, $\Por_0\lessdot_M\Por_\vee$ and $\Por_1\lessdot\Por_\vee$, say that the system $\langle\Por_\wedge,\Por_0,\Por_1,\Por_\vee\rangle$ is \emph{correct with respect to $M$} if, for any $p\in\Por_\wedge$ and $q\in\Por_0$, if $p$ is a reduction of $q$, then $p$ is a reduction of $q$ with respect to $\Por_1,\Por_\vee$.
\end{definition}

The results of this section are applications of this notion to two-step iterations, quotients and direct limits of posets.

\begin{lemma}\label{2stepitemb}
   Let $\Por\in M$, $\Por'\in N$ posets such that $\Por\lessdot_M\Por'$. If $\Qnm\in M$ is a $\Por$-name of a poset, $\Qnm'\in N$ a $\Por'$-name of a poset
   and $\Por'$ forces (with respect to $N$) that $\Qnm\lessdot_{M^\Por}\Qnm'$, then $\Por\ast\Qnm\lessdot_M\Por'\ast\Qnm'$. Also, $\langle\Por,\Por\ast\Qnm,\Por',\Por'\ast\Qnm'\rangle$ is a correct system with respect to $M$.
\end{lemma}
\begin{proof}
  First prove that, if $(p_0,\dot{q}_0),(p_1,\dot{q}_1)\in\Por\ast\Qnm$ are compatible in $\Por'\ast\Qnm'$, then they are also compatible in $\Por\ast\Qnm$. Let $(p',\dot{q}')\in\Por'\ast\Qnm'$ be a common extension. Find $A\in M$ a maximal antichain in $\Por$ contained in $\{p\in\Por\ /\ p\leq p_0,p_1\textrm{\ or\ }p_0\perp p\textrm{\ or\ }p_1\perp p\}$. As $A$ is also maximal antichain in $\Por'$, there exists a $p_2\in A$ compatible with $p'$. $p_2$ is a common extension of $p_0,p_1$ because $p'$ is a common extension of $p_0,p_1$. Also, $p_2$ cannot force, with respect to $\Por$ and $M$, that $\dot{q}_0\perp\dot{q}_1$ because $p'$ forces their compatibility with respect to $\Por'$ and $N$. Therefore, there exists $p\leq p_2$ that forces $\dot{q}_0,\dot{q}_1$ compatible.

  Now, let $\{(p_\alpha,q_\alpha)\ /\ \alpha<\delta\}\in M$ a maximal antichain in $\Por\ast\Qnm$. We claim first that $\Por$ forces that $\{q_\alpha\ /\ p_\alpha\in\dot{G},\alpha<\delta\}$ is a maximal antichain in $\Qnm$, where $\dot{G}$ is a $\Por$-name of its generic subset. Indeed, let $p\in\Por$ be arbitrary and $\dot{q}$ be a $\Por$-name for a condition in $\Qnm$, For some $\alpha<\delta$, there exists a common extension $(r,\dot{s})$ of $(p,\dot{q}),(p_\alpha,\dot{q}_\alpha)$, so $r$ forces that $p_\alpha\in\dot{G}$ and that $\dot{q}_\alpha,\dot{q}$ are compatible.

  Let $(p',\dot{q}')\in\Por'\ast\Qnm'$. Clearly, $p'$ forces (with respect to $\Qor,N$) that $\{\dot{q}_\alpha\ /\ p_\alpha\in\dot{H},\alpha<\delta\}$ is a maximal antichain in $\Qnm'$, where $\dot{H}$ is the $\Por'$-name of its generic subset. Hence, there are $\alpha<\delta$ and $p''\leq p'$ in $\Por'$ that forces $p_\alpha\in\dot{H}$ and $\dot{q}'$ compatible with $\dot{q}_\alpha$. Therefore, $(p',\dot{q}')$ is compatible with $(p_\alpha,\dot{q}_\alpha)$.
\end{proof}

If $\Por$ and $\Qor$ are posets and $\Por\lessdot\Qor$, recall that the quotient $\Qor/\Por$ is defined as a $\Por$-name of the poset $\{q\in\Qor\ /\ \exists_{p\in\dot{G}}(p\textrm{\ is a reduction of }q)\}$ with the order inherited from $\Qor$. It is known that $\Qor\simeq\Por\ast(\Qor/\Por)$.

\begin{lemma}\label{QuotEmb}
   Let $\langle\Por,\Qor,\Por',\Qor'\rangle$ be a correct system. Then, $\Por'$ forces that $\Qor/\Por\lessdot_{V^\Por}\Qor'/\Por'$.
\end{lemma}
\begin{proof}
   Correctness implies directly that $\Vdash_{\Por'}\Qor/\Por\subseteq\Qor'/\Por'$. We prove first that $\Por'$ forces that any pair of incompatible conditions in $\Qor/\Por$ are incompatible in $\Qor'/\Por'$. Let $p'\in\Por'$, $q_0,q_1\in\Qor$ and $q'\in\Qor'$ be such that $p'\Vdash_{\Por'}``q_0,q_1\in\Qor/\Por\textrm{,\ }q'\in\Qor'/\Por'\textrm{\ and\ }q'\leq q_0,q_1"$. We need to find a $p''\leq p'$ in $\Por'$ which forces that $q_0$ and $q_1$ are compatible in $\Qor/\Por$. As $p'\Vdash_{\Por'} q'\in\Qor'/\Por'$, $p'$ is a reduction of $q'$. Find $p\in\Por$ and $q\in\Qor$ such that $q\leq q_0,q_1$, $p$ is a reduction of $q$, $p$ is a reduction of $p'$ and $q$ is a reduction of $q'$. Indeed, choose $p_0\in\Por$ a reduction of $p'$. Then, as $p_0$ is also a reduction of $q'$, there exists a $q''\in\Qor'$ such that $q''\leq q',p_0$. Then, we can find $q\in\Qor$ a reduction of $q''$ such that $q\leq q_0,q_1,p_0$. Now, find $p\leq p_0$ in $\Por$ such that it is a reduction of $q$. Clearly, $p$ and $q$ are as desired. Now, $p\Vdash_\Por q\in\Qor/\Por$ and, as it is a reduction of $p'$, find $p''\in\Por'$ such that $p''\leq p,p'$. Thus, $p''\Vdash_{\Por'}``q\in\Qor/\Por"$ and $q\leq q_0,q_1$.

   Now, let $\dot{A}$ be a $\Por$-name for a maximal antichain in $\Qor/\Por$. Given $p'\in\Por'$ and $q'\in\Qor'$ such that $p'\Vdash_{\Por'}q'\in\Qor'/\Por'$, we need to find $p''\leq p'$ in $\Por'$ and $q\in\Qor$ such that $p''$ forces that $q\in\dot{A}$ and that it is compatible with $q'$ in $\Qor'/\Por'$. Clearly, $p'$ is a reduction of $q'$, so there exists $q''\in\Qor'$ that extends both $p'$ and $q'$. Now, let $q_2\in\Qor$ be a reduction of $q''$. Hence, as
   $\dot{A}$ is the $\Por$-name of a maximal antichain in $\Qor/\Por$, there exist $q,q_3\in\Qor$ and $p\in\Por$ such that $q_3\leq q,q_2$ and $p$ is a reduction of $q_3$ that forces $q\in\dot{A}$. Find $q_4\in\Qor$ such that $q_4\leq p,q_3$. As $q_4\leq q_2$, there exists $q'''\in\Qor'$ extending $q''$ and $q_4$. Now, let $p''\in\Por'$ be a reduction of $q'''$ such that $p''\leq p,p'$. Thus, $p''$ forces that $q\in\dot{A}$, $q'''\in\Qor'/\Por'$ and
   $q'''\leq q,q'$.
\end{proof}

\begin{corollary}\label{CorrQuotEmb}
   Let $\langle\Por,\Qor,\Por',\Qor'\rangle$ and $\langle\Qor,\Ror,\Qor',\Ror'\rangle$ be correct systems. Then, $\Por'$ forces that the system $\langle\Qor/\Por,\Ror/\Por,\Qor'/\Por',\Ror'/\Por'\rangle$ is correct with respect to $V^\Por$.
\end{corollary}
\begin{proof}
   By Lemma \ref{QuotEmb} we only need to prove correctness (to get, e.g., $\Vdash_\Por\Qor/\Por\lessdot\Ror/\Por$, note that $\langle\Por,\Qor,\Por,\Ror\rangle$ is a correct system). In $V^{\Por'}$, we know that $\Ror/\Por\simeq(\Qor/\Por)\ast(\Ror/\Qor)$ and $\Ror'/\Por'\simeq(\Qor'/\Por')\ast(\Ror'/\Qor')$. As $\Qor/\Por\lessdot_{V^\Por}\Qor'/\Por'$ and $\Qor'/\Por'$ forces that $\Ror/\Qor\lessdot_{V^\Qor}\Ror'/\Qor'$ by Lemma \ref{QuotEmb}, we get the correctness we are looking for from Lemma \ref{2stepitemb}.
\end{proof}

Recall that a partial order $\langle I,\leq\rangle$ is \emph{directed} iff any two elements of $I$ have an upper bound in $I$. A sequence of posets $\langle\Por_i\rangle_{i\in I}$ is a \emph{directed system of posets} if, for any $i< j$ in $I$, $\Por_i\lessdot\Por_j$. In this case, the \emph{direct limit of $\langle\Por_i\rangle_{i\in I}$} is defined as the partial order\footnote{In a more general way, we can think of a directed system with complete embeddings $e_{i,j}:\Por_i\to\Por_j$ for $i<j$ in $I$ such that, for $i<j<k$, $e_{j,k}\circ e_{i,j}=e_{i,k}$. This allows to define a direct limit of the system as well.} $\limdir_{i\in I}\Por_i:=\bigcup_{i\in I}\Por_i$. It is clear that, for any $i\in I$, $\Por_i$ is a complete suborder of this direct limit.

\begin{lemma}[Embeddability of direct limits {\cite{Br-Amalg}}, see also
{\cite[Lemma 1.2]{Br-Luminy}}]\label{dirlimEmb}
   Let $I\in M$ be a directed set, $\langle\Por_i\rangle_{i\in I}\in M$ and $\langle\Qor_i\rangle_{i\in I}\in N$ directed systems of posets such that
   \begin{enumerate}[(i)]
      \item for each $i\in I$, $\Por_i\lessdot_M\Qor_i$ and
      \item whenever $i\leq j$, $\langle\Por_i,\Por_j,\Qor_i,\Qor_j\rangle$ is a correct system with respect to $M$
   \end{enumerate}
   Then, $\Por:=\limdir_{i\in I}\Por_i$ is a complete suborder of $\Qor:=\limdir_{i\in I}\Qor_i$ with respect to $M$ and, for any $i\in I$, $\langle\Por_i,\Por,\Qor_i,\Qor\rangle$ is a correct system with respect to $M$.
\end{lemma}
\begin{proof}
   Let $A\in M$ be a maximal antichain of $\Por$. Let $q\in\Qor$, so there is some $i\in I$ such that $q\in\Qor_i$. Work within $M$. Enumerate $A:=\{p_\alpha\ /\ \alpha<\delta\}$ for some ordinal $\delta$ and, for each $\alpha<\delta$, choose $j_\alpha\geq i$ in $I$ such that $p_\alpha\in\Por_{j_\alpha}$. Now, if $p\in\Por_i$, there is some $\alpha<\delta$ such that $p$ is compatible with $p_\alpha$ in $\Por_{j_\alpha}$, so there exists $p'\leq p$ which is a reduction of $p_\alpha$ with respect to $\Por_i,\Por_{j_\alpha}$.

   The previous density argument implies, in $N$, that $q$ is compatible with some $p\in\Por_i$ which is a reduction of $p_\alpha$ for some $\alpha<\delta$. By \ \!(ii), $p$ is a reduction of $p_\alpha$ with respect to $\Qor_i,\Qor_{j_\alpha}$, which implies that $q$ is compatible with $p_\alpha$.
\end{proof}

\begin{lemma}\label{dirlimquot}
   Let $\langle\Por_i\rangle_{i\in I}$ be a directed system of posets, $\Por$ its direct limit. Assume that $\Qor$ is a complete suborder of $\Por_i$ for all $i\in I$. Then, $\Qor$ forces that $\Por/\Qor=\limdir_{i\in I}\Por_i/\Qor$.
\end{lemma}
\begin{proof}
   For $i\in I$, as $\langle\Qor,\Por_i,\Qor,\Por\rangle$ is a correct system, by Lemma \ref{QuotEmb} $\Qor$ forces that $\Por_i/\Qor$ is a complete suborder of $\Por/\Qor$. It is easy to see that $\Qor$ forces $\Por/\Qor=\bigcup_{i\in I}\Por_i/\Qor$.
\end{proof}

\section{Preservation properties}\label{SecPresProp}

Fix, for this section, an uncountable regular cardinal $\theta$ and a cardinal $\lambda\geq\theta$.

\begin{context}[{\cite{Go-Tools},\cite[Sect. 6.4]{BarJu}}]\label{ContextPres}
   Fix $\langle\sqsubset_n\rangle_{n<\omega}$ an increasing sequence of 2-place closed relations in $\omega^\omega$ such that, for any $n<\omega$ and $g\in\omega^\omega$, $(\sqsubset_n)^g=\{f\in\omega^\omega\ /\ f\sqsubset_n g\}$ is (closed) nowhere dense.

   For $f,g\in\omega^\omega$, say that \emph{$g$ $\sqsubset$-dominates $f$} if $f\sqsubset g$. $F\subseteq\omega^\omega$ is a \emph{$\sqsubset$-unbounded family} if no function in $\omega^\omega$ dominates all the members of $F$. Associate with this notion the cardinal $\bfrak_\sqsubset$, which is the least size of a $\sqsubset$-unbounded family. Dually, say that $C\subseteq\omega^\omega$ is a \emph{$\sqsubset$-dominating family} if any real in $\omega^\omega$ is dominated by some member of $C$. The cardinal $\dfrak_\sqsubset$ is the least size of a $\sqsubset$-dominating family. For a set $Y$ and a real $f\in\omega^\omega$, say that \emph{$f$ is $\sqsubset$-unbounded over $Y$} if $\forall_{g\in\omega^\omega\cap Y}(f\not\sqsubset g)$, which we denote by $f\not\sqsubset Y$.
\end{context}

Although this context is defined for $\omega^\omega$, the domain and codomain of $\sqsubset$ can be any uncountable Polish space coded by reals in $\omega^\omega$.

\begin{example}\label{ExBdRelation}
  \begin{enumerate}[(1)]
     \item For $n<\omega$ and $f,g\in\omega^\omega$, $f\leq^*_n g$ denotes $\forall_{k\geq n}(f(k)\leq g(k))$, so $f\leq^*g\sii\forall_{n<\omega}^\infty(f(n)\leq g(n))$. The \emph{(un)bounding number} is defined as $\bfrak:=\bfrak_{\leq^*}$ and the \emph{dominating number} is $\dfrak:=\dfrak_{\leq^*}$, which are classical cardinal invariants.
     \item For $n<\omega$ and $A,B\in[\omega]^\omega$, define $A\propto_n B\sii(B\menos n\subseteq A\textrm{\ or }B\menos n\subseteq\omega\menos A)$, so $A\propto B$ iff either $B\subseteq^* A$ or $B\subseteq^*\omega\menos A$, where $X\subseteq^*Y$ means that $Y\menos X$ is finite. Note that $A\not\propto B$ iff \emph{$A$ splits $B$}, that is, $A\cap B$ and $B\menos A$ are infinite. The \emph{splitting number} is defined as $\sfrak:=^*\bfrak_\propto$ and $\rfrak:=\dfrak_\propto$ is the \emph{(un)reaping number}, which are also classical cardinal invariants.
     \item Consider, for $f,g\in\omega^\omega$ and $n<\omega$, $f=^*_n g$ defined as $\forall_{k\geq n}(f(n)=g(n))$. Then, $f=^* g$ iff $\forall^\infty_{k<\omega}(f(k)=g(k))$. Note that $\bfrak_{=^*}=2$ and $\dfrak_{=^*}=\cfrak$.

         Here, the associated cardinal invariants are not that important. We are interested in the meaning of ``$f\in\omega^\omega$ is $=^*$-unbounded over $M$", which is equivalent to $f\notin M$ when $M$ is a model of some finite subset of axioms of $\thzfc$.
  \end{enumerate}
\end{example}

\begin{lemma}\label{b_sqsubset leq nonM}
   $\bfrak_\sqsubset\leq\non(\Mwf)$ and $\cov(\Mwf)\leq\dfrak_\sqsubset$.
\end{lemma}
\begin{proof}
   Immediate from the fact that $(\sqsubset)^g$ is meager for any $g\in\omega^\omega$.
\end{proof}

\begin{definition}\label{DefStrUnb}
  Let $F\subseteq\omega^\omega$. Say that $F$ is \emph{$\theta$-$\sqsubset$-unbounded} if, for any $X\subseteq\omega^\omega$ of size $<\theta$, there is an $f\in F$ such that $f\not\sqsubset X$.
\end{definition}

Clearly, any $\theta$-$\sqsubset$-unbounded family is $\sqsubset$-unbounded, so

\begin{lemma}\label{StrUnb-b leq}
   If $F\subseteq\omega^\omega$ is $\theta$-$\sqsubset$-unbounded, then $\bfrak_\sqsubset\leq|F|$ and $\theta\leq\dfrak_\sqsubset$.
\end{lemma}

The following is a property that expresses when a forcing notion preserves $\theta$-$\sqsubset$-unbounded families of the ground model.

\begin{definition}[Judah and Shelah {\cite{JuSh-KunenMillerchart}}, {\cite[Def. 6.6.4]{BarJu}}]\label{DefGood}
   A forcing notion $\Por$ is \emph{$\theta$-$\sqsubset$-good} if the following property holds\footnote{According to \cite[Def. 6.6.4]{BarJu}, our property is called \emph{really $\theta$-$\sqsubset$-good} while $\theta$-$\sqsubset$-good stands for another property. However, \cite[Lemma 6.6.5]{BarJu} states that really $\theta$-$\sqsubset$-good implies $\theta$-$\sqsubset$-good, and it is also easy to see that the converse is true for $\theta$-cc posets, see details in \cite[Lemma 2]{Me-MatIt}.}: For any $\Por$-name $\dot{h}$ for a real in $\omega^\omega$, there exists a nonempty $Y\subseteq\omega^\omega$ (in the ground model) of size $<\theta$ such that, for any $f\in\omega^\omega$, if $f\not\sqsubset Y$ then $\Vdash f\not\sqsubset\dot{h}$.

   Say that $\Por$ is \emph{$\sqsubset$-good} if it is $\aleph_1$-$\sqsubset$-good\footnote{In \cite{Me-MatIt,Me-TempIt}, ``$\Por$ is $\theta$-$\sqsubset$-good" is denoted by $(+^\theta_{\Por,\sqsubset})$ and ``$\Por$ is $\sqsubset$-good" is denoted by $(+_{\Por,\sqsubset})$}.
\end{definition}

Note that $\theta<\theta'$ implies that any $\theta$-$\sqsubset$-good poset is $\theta'$-$\sqsubset$-good. Also, if $\Por\lessdot\Qor$ and $\Qor$ is $\theta$-$\sqsubset$-good, then $\Por$ is $\theta$-$\sqsubset$-good.

\begin{example}\label{ExGood}
  \begin{enumerate}[(1)]
     \item Miller \cite{Miller} proved that $\Eor$, the canonical forcing that adds an eventually different real, is $\leq^*$-good. Also, any $\omega^\omega$-bounding poset is $\leq^*$-good, in particular, random forcing.
     \item Baumgartner and Dordal \cite{BaumDor-Dom} proved that Hechler forcing $\Dor$ (the canonical forcing that adds a dominating real) is $\propto$-good. See also {\cite[Lemma 3.8]{Br-Bog}} for a proof.
     \item Any $\theta$-cc poset is $\theta$-$=^*$-good. In particular, any ccc poset is $=^*$-good.
     \item Any poset of size $<\theta$ is $\theta$-$\sqsubset$-good. In particular, Cohen forcing $\Cor$ is $\sqsubset$-good. For a proof, see \cite[Thm 6.4.7]{BarJu}, also \cite[Lemma 4]{Me-MatIt}.
  \end{enumerate}
\end{example}

\begin{lemma}[{\cite[Lemma 6.4.8]{BarJu}}, see also {\cite[Lemma 3]{Me-MatIt}}]\label{GoodPresStrUnb}
   Assume that $\Por$ is $\theta$-$\sqsubset$-good.
   \begin{enumerate}[(a)]
      \item If $F\subseteq\omega^\omega$ is $\theta$-$\sqsubset$-unbounded, then $\Por$ forces that $F$ is still $\theta$-$\sqsubset$-unbounded.
      \item If $\dfrak_\sqsubset\geq\lambda$, then $\Por$ forces that $\dfrak_\sqsubset\geq\lambda$.
   \end{enumerate}
\end{lemma}

Judah and Shelah \cite{JuSh-KunenMillerchart} proved that $\theta$-$\sqsubset$-goodness is preserved in fsi of $\theta$-$\sqsubset$-good $\theta$-cc posets. We generalize the preservation in the limits steps in Theorem \ref{PresGoodLimdir}.

\begin{lemma}[{\cite[Lemma 6.4.11]{BarJu}}]\label{PresGood2step}
   Let $\Por$ be a poset and $\Qnm$ a $\Por$-name for a poset. If $\Por$ is $\theta$-cc, $\theta$-$\sqsubset$-good and $\Por$ forces that $\Qnm$ is $\theta$-$\sqsubset$-good, then $\Por\ast\Qnm$ is $\theta$-$\sqsubset$-good.
\end{lemma}

\begin{theorem}[Preservation of goodness in short direct limits]\label{PresGoodLimdir}
   Let $I$ be a directed partial order, $\langle\Por_i\rangle_{i\in I}$ a directed system and $\Por=\limdir_{i\in I}\Por_i$. If $|I|<\theta$ and $\Por_i$ is $\theta$-$\sqsubset$-good for any $i\in I$, then $\Por$ is $\theta$-$\sqsubset$-good.
\end{theorem}
\begin{proof}
   Let $\dot{h}$ be a $\Por$-name for a real in $\omega^\omega$. For $i\in I$, find a $\Por_i$-name for a real $\dot{h}_i$ and a sequence $\{\dot{p}^i_m\}_{m<\omega}$ of $\Por_i$-names that represents a decreasing sequence of conditions in $\Por/\Por_i$ such that $\Por_i$ forces that $\dot{p}^i_m\Vdash_{\Por/\Por_i}\dot{h}\frestr m=\dot{h}_i\frestr m$. For each $i\in I$ choose $Y_i\subseteq\omega^\omega$ of size $<\theta$ that witnesses goodness of $\Por_i$ for $\dot{h}_i$. As $|I|<\theta$, $Y=\bigcup_{i\in I}Y_i$ has size $<\theta$ by regularity of $\theta$.

   We prove that $Y$ witnesses goodness of $\Por$ for $\dot{h}$. Assume, towards a contradiction, that $f\in\omega$, $f\not\sqsubset Y$ and that there are $p\in\Por$ and $n<\omega$ such that $p\Vdash_\Por f\sqsubset_n\dot{h}$. Choose $i\in I$ such that $p\in\Por_i$. Let $G$ be $\Por_i$-generic over the ground model $V$ with $p\in G$. Then, by the choice of $Y_i$, $f\not\sqsubset h_i$, in particular, $f\not\sqsubset_n h_i$. As $C:=(\sqsubset_n)_f=\{g\in\omega^\omega\ /\ f\sqsubset_n g\}$ is closed, there is an $m<\omega$ such that $[h_i\frestr m]\cap C=\varnothing$. Thus, $p^i_m\Vdash_{\Por/\Por_i}[\dot{h}\frestr m]\cap C=\varnothing$, that is, $p^i_m\Vdash_{\Por/\Por_i}f\not\sqsubset_n\dot{h}$. On the other hand, by hypothesis, $\Vdash_{\Por/\Por_i}f\sqsubset_n\dot{h}$, a contradiction.
\end{proof}

\begin{corollary}[Judah and Shelah {\cite{JuSh-KunenMillerchart}}, Preservation of goodness in well ordered direct limits]\label{PresGoodwoLimdir}
   Let $\delta$ be a limit ordinal and $\{\Por_\alpha\}_{\alpha<\delta}$ be a sequence of posets such that, for $\alpha<\beta<\delta$, $\Por_\alpha\lessdot\Por_\beta$. If $\Por_\delta=\limdir_{\alpha<\delta}\Por_\alpha$ is $\theta$-cc and $\Por_\alpha$ is $\theta$-$\sqsubset$-good for any $\alpha<\delta$, then $\Por_\delta$ is $\theta$-$\sqsubset$-good.
\end{corollary}
\begin{proof}
   First assume that $\cf(\delta)<\theta$, so there is an increasing sequence $\{\alpha_\xi\}_{\xi<\cf(\delta)}$ that converges to $\delta$. Then, $\Por_\delta=\limdir_{\xi<\cf(\delta)}\Por_{\alpha_\xi}$, which implies that $\Por_\delta$ is $\theta$-$\sqsubset$-good by Theorem \ref{PresGoodLimdir}.

   Now, assume that $\cf(\delta)\geq\theta$. Let $\dot{h}$ be a $\Por_\delta$-name for a real. By $\theta$-cc, there is an $\alpha<\theta$ such that $\dot{h}$ is a $\Por_\alpha$-name. Then, by hypothesis, there is $Y\subseteq\omega^\omega$ of size $<\theta$ that witnesses goodness of $\Por_\alpha$ for $\dot{h}$. It is clear that $Y$ also witnesses goodness of $\Por_\delta$.
\end{proof}

\begin{corollary}[Judah and Shelah {\cite{JuSh-KunenMillerchart}}, Preservation of goodness in fsi {\cite[Lemma 6.4.12]{BarJu}}]\label{PresGoodFsi}
   Let $\Por_\delta=\langle\Por_\alpha,\Qnm_\alpha\rangle_{\alpha<\delta}$ be a fsi of $\theta$-cc forcing notions. If, for each $\alpha<\delta$, $\Por_\alpha$ forces that $\Qnm_\alpha$ is $\theta$-$\sqsubset$-good, then $\Por_\delta$ is $\theta$-$\sqsubset$-good.
\end{corollary}
\begin{proof}
   Prove by induction on $\alpha\leq\delta$ that $\Por_\alpha$ is $\theta$-$\sqsubset$-good. Step $\alpha=0$ is trivial, successor step comes from Lemma \ref{PresGood2step} and the limit step is a direct consequence of Corollary \ref{PresGoodwoLimdir}.
\end{proof}

Beyond the applications on fsi, Theorem \ref{PresGoodLimdir} can be applied to obtain goodness in template iterations, for example, see \cite[Thm. 4.13 and 4.15]{Me-TempIt}.

The following results show how to add $\sqsubset$-unbounded families with Cohen reals, in order to get values for $\bfrak_\sqsubset$ and $\dfrak_\sqsubset$.

\begin{lemma}\label{AddUnbGenIt}
   Let $\nu$ be an uncountable regular cardinal, $\langle\Por_\alpha\rangle_{\alpha<\nu}$ a $\lessdot$-increasing sequence of forcing notions and $\Por_\nu=\limdir_{\alpha<\nu}\Por_\alpha$. If
   \begin{enumerate}[(i)]
      \item for each $\alpha<\nu$, $\Por_{\alpha+1}$ adds a Cohen real over $V^{\Por_\alpha}$, and
      \item $\Por_\nu$ is ccc,
   \end{enumerate}
   then, $\Por_\nu$ adds a $\nu$-$\sqsubset$-unbounded family (of Cohen reals) of size $\nu$. Moreover, it forces $\bfrak_\sqsubset\leq\nu$ and $\nu\leq\dfrak_\sqsubset$.
\end{lemma}
\begin{proof}
   Let $\dot{c}_\alpha$ be a $\Por_{\alpha+1}$-name of a Cohen real over $V^{\Por_\alpha}$. Then, $\Por_\nu$ forces that $\{\dot{c}_\alpha\ /\ \alpha<\nu\}$ is a $\nu$-$\sqsubset$-unbounded family. Indeed, if $\{\dot{x}_\xi\}_{\xi<\mu}$ is a sequence of $\Por_\nu$-names for reals with $\mu<\nu$, by (ii) there is an $\alpha<\nu$ such that $\{\dot{x}_\xi\}_{\xi<\mu}$ is a sequence of $\Por_\alpha$-names, so $\Por_{\alpha+1}$ forces that $\dot{c}_\alpha\not\sqsubset\dot{x}_\xi$ for all $\xi<\mu$. This last assertion holds because $(\sqsubset)^g$ is an $F_\sigma$ meager set for any $g\in\omega^\omega$ (see Context \ref{ContextPres}).

   The second statement is a consequence of Lemma \ref{StrUnb-b leq}.
\end{proof}

\begin{lemma}\label{AddsmallUnb}
   Let $\delta\geq\theta$ be an ordinal and $\Por_\delta=\langle\Por_\alpha,\Qnm_\alpha\rangle_{\alpha<\delta}$ be a fsi such that,
   \begin{enumerate}[i)]
      \item for $\alpha<\theta$, $\Qnm_\alpha$ is forced (by $\Por_\alpha$) to be ccc and to have two incompatible conditions, and
      \item for $\theta\leq\alpha<\delta$, $\Qnm_\alpha$ is forced to be $\theta$-cc and  $\theta$-$\sqsubset$-good.
   \end{enumerate}
   Then,
   \begin{enumerate}[(a)]
      \item $\Por_\theta$ adds a $\theta$-$\sqsubset$-unbounded family (of Cohen reals) of size $\theta$.
      \item The family added in (a) is forced to be a $\theta$-$\sqsubset$-unbounded family by $\Por_\delta$. In particular, it forces that $\bfrak_\sqsubset\leq\theta\leq\dfrak_\sqsubset$.
   \end{enumerate}
\end{lemma}
\begin{proof}
   \begin{enumerate}[(a)]
      \item This is a direct consequence of Lemma \ref{AddUnbGenIt} and the fact that this iteration adds Cohen reals at limit stages.
      \item Let $\dot{C}$ be a $\Por_\theta$-name for a family of reals as in (a). Step in $V_\theta$. Note that $\Por_\delta/\Por_\theta$ is equivalent to the fsi $\langle\Por_\alpha/\Por_\theta,\Qnm_\alpha\rangle_{\theta\leq\alpha<\delta}$. Thus, by Corollary \ref{PresGoodFsi}, $\Por_\delta/\Por_\theta$ is $\theta$-$\sqsubset$-good. Hence, by Lemma \ref{GoodPresStrUnb}, it forces that $C$ is $\theta$-$\sqsubset$-unbounded.
   \end{enumerate}
\end{proof}

\begin{example}
   Baumgartner and Dordal \cite{BaumDor-Dom} proved that it is consistent that $\sfrak<\bfrak$. This is done in the following way. Fix a regular cardinal $\mu>\aleph_1$ and let $\Por$ be the poset resulting by a fsi of length $\mu$ of Hechler forcing. This adds a scale of length $\mu$, so $\Por$ forces $\bfrak=\dfrak=\mu$. On the other hand, $\Por$ is $\propto$-good because of Example \ref{ExGood}(2) and Corollary \ref{PresGoodFsi} and, by Lemma \ref{AddsmallUnb}, an $\aleph_1$-$\propto$-unbounded family is added at the $\omega_1$ stage of the iteration and it is preserved until the final extension, so $\Por$ forces that $\sfrak=\aleph_1$.

   If $\nu<\mu$ is an uncountable regular cardinal, the construction of the iteration can be modified in order to produce a $\mu$-$\propto$-good poset which forces that any family of size $<\nu$ of infinite subsets of $\omega$ has an $\propto$-upper bound, this by a good keeping argument using Mathias forcing with filter bases of size $<\nu$ (Example \ref{ExGood}(4) is also used for this).
\end{example}

From now on in this section, fix $M\subseteq N$ transitive models of $\thzfc$. We discuss a property of preserving unbounded reals over $M$ along parallel iterations from $M$ and $N$. The remaining results of this section are based on \cite{BlSh-MatrIt}, \cite{BrFi-MatrIt} and \cite{Me-MatIt}.

Consider $\sqsubset$ from Context \ref{ContextPres} with parameters in $M$ and fix $c\in N$ a $\sqsubset$-unbounded real over $M$. As Cohen reals over $M$ that belong to $N$ are $\sqsubset$-unbounded over $M$, typically $c$ is such a real.

\begin{definition}\label{DefPresUnbreal}
   Let $\Por\in M$ and $\Qor\in N$ be posets such that $\Por\lessdot_M\Qor$. Consider the property
   \[(\star,\Por,\Qor,M,N,\sqsubset,c):\ \ \ \textrm{for every $\dot{h}\in M$ $\Por$-name for a real, $\Vdash_{\Qor,N}c\not\sqsubset\dot{h}$.}\]
   This means that $c$ is forced by $\Qor$ (in $N$) to be $\sqsubset$-unbounded over $M^\Por$.
\end{definition}

As an example, we have

\begin{lemma}\label{PresUnbrealEx}
   \begin{enumerate}[(a)]
      \item (\cite[Thm. 7]{Me-MatIt}) Let $\Sor$ be a Suslin ccc poset with parameters in $M$. If $\Sor$ is $\sqsubset$-good in $M$, then $(\star,\Sor^M,\Sor^N,M,N,\sqsubset,c)$ holds.
      \item (\cite[Lemma 11]{BrFi-MatrIt}) Let $\Por\in M$ be a poset. Then, $(\star,\Por,\Por,M,N,\sqsubset,c)$ holds.
   \end{enumerate}
\end{lemma}

\begin{lemma}\label{PresUnbreal2step}
   Let $\Por\in M$, $\Por'\in N$ posets such that $(\star,\Por,\Por',M,N,\sqsubset,c)$ holds. Also, let $\Qnm\in M$ be a $\Por$-name of a poset and $\Qnm'\in N$ a $\Por'$-name of a poset such that $\Por'$ forces (with respect to $N$) that $(\star,\Qnm,\Qnm',M^\Por,N^{\Por'},\sqsubset,c)$. Then $(\star,\Por\ast\Qnm,\Por'\ast\Qnm',M,N,\sqsubset,c)$ holds.
\end{lemma}
\begin{proof}
   From Lemma \ref{2stepitemb} it is clear that $\Por\ast\Qnm\lessdot_M\Por'\ast\Qnm'$. $(\star,\Por,\Por',M,N,\sqsubset,c)$ indicates that $\Vdash_{\Por',N}c\not\sqsubset M^\Por$ and, as it forces $(\star,\Qnm,\Qnm',M^\Por,N^{\Por'},\sqsubset,c)$, then $\Vdash_{\Por',N}$``$\Vdash_{\Qnm',N^{\Por'}}c\not\sqsubset M^{\Por\ast\Qnm'}$".
\end{proof}

The following result is a generalization of the corresponding fact (originally proved by Blass and Shelah \cite{BlSh-MatrIt}) for finite support iterations (Corollary \ref{PresUnbrealFsi}). The proof is almost the same (see, for example, \cite[Lemma 12]{BrFi-MatrIt}).

\begin{theorem}\label{PresUnbrealLimit}
   Let $I\in M$ be a directed set, $\langle\Por_i\rangle_{i\in I}\in M$ and $\langle\Qor_i\rangle_{i\in I}\in N$ directed systems of posets such that
   \begin{enumerate}[(i)]
      \item for each $i\in I$, $(\star,\Por_i,\Qor_i,M,N,\sqsubset,c)$ holds and
      \item whenever $i\leq j$, $\langle\Por_i,\Por_j,\Qor_i,\Qor_j\rangle$ is a correct system with respect to $M$
   \end{enumerate}
   Then, $(\star,\Por,\Qor,M,N,\sqsubset,c)$ where $\Por:=\limdir_{i\in I}\Por_i$ and $\Qor:=\limdir_{i\in I}\Qor_i$. Moreover, for any $i\in I$, $\langle\Por_i,\Por,\Qor_i,\Qor\rangle$ is a correct system with respect to $M$.
\end{theorem}
\begin{proof}
   By Lemma \ref{dirlimEmb}, it is enough to prove that, if \ \!$\dot{h}\in M$ is a $\Por$-name for a real in $\omega^\omega$, then $\Vdash_{\Qor,N}c\not\sqsubset\dot{h}$. Assume, towards a contradiction, that there are $q\in\Qor$ and $n<\omega$ such that $q\Vdash_{\Qor,N}c\sqsubset_n\dot{h}$. Choose $i\in I$ such that $q\in\Qor_i$.

   Let $G$ be $\Qor_i$-generic over $N$ with $q\in G$. By assumption, $\Vdash_{\Qor/\Qor_i,N[G]}c\sqsubset_n\dot{h}$. In $M[G\cap\Por]$, find $g\in\omega^\omega$ and a decreasing chain $\{p_k\}_{k<\omega}$ in $\Por/\Por_i$ such that $p_k\Vdash_{\Por/\Por_i,M[G\cap\Por]}\dot{h}\frestr k=g\frestr k$. In $N[G]$, by hypothesis, $c\not\sqsubset g$, so there is a $k<\omega$ such that $[g\frestr k]\cap(\sqsubset_n)_c=\varnothing$. Then, as $\Por/\Por_i\lessdot_{M[G\cap\Por]}\Qor/\Qor_i$ by Lemma \ref{QuotEmb}, $p_k\Vdash_{\Qor/\Qor_i,N[G]}[\dot{h}\frestr k]\cap(\sqsubset_n)_c=\varnothing$, that is, $p_k\Vdash_{\Qor/\Qor_i,N[G]} c\not\sqsubset_n\dot{h}$, which is a contradiction.
\end{proof}

\begin{corollary}[Blass and Shelah {\cite{BlSh-MatrIt}}]\label{PresUnbrealFsi}
   Let $\Por_\delta=\langle\Por_\alpha,\Qnm_\alpha\rangle$ be a fsi in $M$ and $\Por'_\delta=\langle\Por'_\alpha,\Qnm'_\alpha\rangle$ a fsi in $N$. Assume that, for any $\alpha<\delta$, if $\Por_\alpha\lessdot_M\Por'_\alpha$ and $\Por'_\alpha$ forces (in $N$) $(\star,\Qnm_\alpha,\Qnm'_\alpha,M^{\Por_\alpha},N^{\Por'_\alpha},\sqsubset,c)$. Then, $(\star,\Por_\alpha,\Por'_\alpha,M,N,\sqsubset,c)$ holds for any $\alpha\leq\delta$.
\end{corollary}

\section{Forcing with ultrapowers}\label{SecUltrapow}

We present some facts, introduced by Shelah \cite{Sh-TempIt} (see also \cite{Br-TempIt}) about forcing with the ultrapower of a ccc poset by a measurable cardinal.

Recall that a cardinal $\kappa$ is \emph{measurable} if it is uncountable and has a $\kappa$-complete (non-trivial) ultrafilter $\Uwf$, where \emph{$\kappa$-complete} means that $\Uwf$ is closed under intersections of $<\kappa$ many sets. Note that, in this case, $\kappa$ is an inaccessible cardinal. For a formula $\varphi(x)$ in the language of $\thzfc$, say that \emph{$\varphi(\alpha)$ holds for $\Dwf$-many $\alpha$} iff $\{\alpha<\kappa\ /\ \varphi(\alpha)\}\in\Dwf$. To fix a notation about \emph{ultraproducts} and \emph{ultrapowers}, if $\langle X_\alpha\rangle_{\alpha<\kappa}$ is a sequence of sets, $(\prod_{\alpha<\kappa}X_\alpha)/\Dwf=[\{X_\alpha\}_{\alpha<\kappa}]$ denotes the quotient of $\prod_{\alpha<\kappa}X_\alpha$ modulo the equivalence relation given by $x\sim_\Dwf y$ iff $x_\alpha=y_\alpha$ for $\Dwf$-many $\alpha<\kappa$. If $x=\langle x_\alpha\rangle_{\alpha<\kappa}\in\prod_{\alpha<\kappa}X_\alpha$, denote its equivalence class under $\sim_\Dwf$ by $\bar{x}=\langle x_\alpha\rangle_{\alpha<\omega}/\Dwf$. It is known that posets of size $<\kappa$ does not destroy the measurability of $\kappa$, that is, preserves the $\kappa$-completeness of $\Dwf$. For facts about measurable cardinals (and large cardinals in general), see \cite{Kanamori}.

Fix a poset $\Por$, a measurable cardinal $\kappa$ and a $\kappa$-complete ultrafilter $\Dwf$ on $\kappa$. For notation, if $p\in\Por^\kappa$, denote $p_\alpha=p(\alpha)$. For $p,q\in\Por^\kappa$ say that $p\leq_\Dwf q$ iff $p_\alpha\leq q_\alpha$ for $\Dwf$-many $\alpha$. The poset $\Por^\kappa/\Dwf$, ordered by $\bar{p}\leq\bar{q}$ iff $p\leq_\Dwf q$, is the \emph{$\Dwf$-ultrapower of $\Por$}.

\begin{lemma}[Shelah {\cite{Sh-TempIt}}, see also {\cite[Lemma 0.1]{Br-TempIt}}]\label{UltraprodEmb}
   Consider $i:\Por\to\Por^\kappa/\Dwf$ defined as $i(r)=\bar{r}$ where $r_\alpha=r$ for all $\alpha<\kappa$. Then, $i$ is a complete embedding iff $\Por$ is $\kappa$-cc.
\end{lemma}

\begin{lemma}[Shelah {\cite{Sh-TempIt}}, see also {\cite[Lemma 0.2]{Br-TempIt}}]\label{Ultraprodccc}
   If $\mu<\kappa$ and $\Por$ is $\mu$-cc, then $\Por^\kappa/\Dwf$ is also $\mu$-cc.
\end{lemma}

Fix a ccc poset $\Por$. We analyze how $\Por^\kappa/\Dwf$-names for reals looks like in terms of $\Por$-names of reals. For reference, consider $\omega^\omega$. First we show how to construct a $\Por^\kappa/\Dwf$-name from a sequence $\langle\dot{f}_\alpha\rangle_{\alpha<\kappa}$ of $\Por$-names of reals. For each $\alpha<\omega$ and $n<\omega$, let $\{p_\alpha^{n,j}\ /\ j<\omega\}$ be a maximal antichain in $\Por$ and $k_\alpha^n:\omega\to\omega$ a function such that\ \! $p_\alpha^{n,j}\Vdash\dot{f}_\alpha(n)=k_\alpha^n(j)$ for all $j<\omega$. Put $p^{n,j}=\langle p^{n,j}_\alpha\rangle_{\alpha<\kappa}$ and note that, for $n<\omega$, $\{\bar{p}^{n,j}\ /\ j<\omega\}$ is a maximal antichain in $\Por^\kappa/\Dwf$ by $\omega_1$-completeness of $\Dwf$. Also, as $\cfrak<\kappa$, there exist a $D\in\Dwf$ and, for each $n<\omega$, a function $k^n:\omega\to\omega$ such that $k^n_\alpha=k^n$ for all $\alpha\in D$. Define $\dot{f}=\langle\dot{f}_\alpha\rangle_{\alpha<\kappa}/\Dwf$ the $\Por^\kappa/\Dwf$-name for a real such that, for any $n,j<\omega$, $\bar{p}^{n,j}\Vdash\dot{f}(n)=k^n(j)$. Note that, if $\langle \dot{g}_\alpha\rangle_{\alpha<\kappa}$ is a sequence of $\Por$-names of reals and $\Vdash_\Por\dot{f}_\alpha=\dot{g}_\alpha$ for $\Dwf$-many $\alpha$, then $\Vdash_{\Por^\kappa/\Dwf}\dot{f}=\dot{g}$ where $\dot{g}=\langle\dot{g}_\alpha\rangle_{\alpha<\kappa}/\Dwf$.

We show that any $\Por^\kappa/\Dwf$-name $\dot{f}$ for a real can be described in this way. For each $n<\omega$, let $A^n:=\{\bar{p}^{n,j}\ /\ j<\omega\}$ be a maximal antichain in $\Por^\kappa/\Dwf$ and $k^n:\omega\to\omega$ such that $\bar{p}^{n,j}\Vdash\dot{f}(n)=k^n(j)$. By $\kappa$-completeness of $\Dwf$, we can find $D\in\Dwf$ such that, for all $\alpha\in D$, $\{p^{n,j}_\alpha\ /\ j<\omega\}$ is a maximal antichain in $\Por$ for any $n<\omega$. Let $\dot{f}_\alpha$ be the $\Por$-name of a real such that $p^{n,j}_\alpha\Vdash_\Por\dot{f}_\alpha=k^n(j)$. For $\alpha\in\kappa\menos D$ just choose any $\Por$-name $\dot{f}_\alpha$ for a real, so we get that $\Vdash_{\Por^\kappa/\Dwf}\dot{f}=\langle\dot{f}_\alpha\rangle_{\alpha<\kappa}/\Dwf$.

\begin{theorem}\label{ProjStatandUltrapow}
   Fix $m<\omega$ and a $\boldsymbol{\Sigma}^1_m$ property $\varphi(x)$ of reals. Let $\langle\dot{f}_{\alpha}\rangle_{\alpha<\kappa}$ be a sequence of $\Por$-names of reals and put $\dot{f}=\langle\dot{f}_{\alpha}\rangle_{\alpha<\kappa}/\Dwf$. Then, for $\bar{p}\in\Por^\kappa/\Dwf$, $\bar{p}\Vdash\varphi(\dot{f})$ iff $p_\alpha\Vdash_\Por\varphi(\dot{f}_{\alpha})$ for $\Dwf$-many $\alpha$.
\end{theorem}
\begin{proof}
   This is proved by induction on $m<\omega$. Recall that $\boldsymbol{\Sigma}^1_0=\boldsymbol{\Pi}^1_0$ corresponds to the pointclass of closed sets. Thus, if $\varphi(x)$ is a $\boldsymbol{\Sigma}^1_0$-property of reals, there exists a tree $T\subseteq\omega^\omega$ such that, for $x\in\omega^\omega$, $\varphi(x)$ iff $x\in[T]:=\{z\in\omega^\omega\ /\ \forall_{k<\omega}(z\frestr k\in T)\}$.

   As in the previous discussion choose, for each $n<\omega$, a maximal antichain $\{\bar{p}^{n,j}\ /\ j<\omega\}$ on $\Por^\kappa/\Dor$ and a function $k^n:\omega\to\omega$ such that $\bar{p}^{n,j}\Vdash\dot{f}(n)=k^n(j)$ and $p^{n,j}_\alpha\Vdash\dot{f}_\alpha(n)=k^n(j)$ for $\Dwf$-many $\alpha$. First, assume that $p_\alpha\Vdash f_\alpha\in[T]$ for $\Dwf$-many $\alpha$ and fix $k<\omega$. If $\bar{q}\leq\bar{p}$, we can find a decreasing sequence $\{\bar{q}^i\}_{i\leq k}$ and a $t\in\omega^k$ such that $\bar{q}^0=\bar{q}$ and $\bar{q}^{i+1}\leq\bar{p}^{i,t(i)}$ for any $i<k$. Therefore, $\bar{q}^k\Vdash\dot{f}\frestr k=k^n\circ t$ and, for $\Dwf$-many $\alpha$, $q^k_\alpha\Vdash\dot{f}_\alpha\frestr k=k^n\circ t$, so $k^n\circ t\in T$.

   Now, assume that $p_\alpha\not\Vdash f_\alpha\in[T]$ for $\Dwf$-many $\alpha$. Without loss of generality, we may assume that there is a $k<\omega$ such that $p_\alpha\Vdash\dot{f}_\alpha\frestr k\notin T$ for $\Dwf$-many $\alpha$. To prove $\bar{p}\Vdash\dot{f}\frestr k\notin T$ repeat the same argument as before, but note that this time we get $k^n\circ t\notin T$.

   For the inductive step, assume that $\varphi(x)$ is $\boldsymbol{\Sigma}^1_{m+1}$, so $\varphi(x)\sii\exists_{y\in\omega^\omega}\psi(x,y)$ where $\psi(x,y)$ is $\boldsymbol{\Pi}^1_m(\omega^\omega\times\omega^\omega)$ (notice that, if this theorem is valid for all $\boldsymbol{\Sigma}^1_m$-statements, then it is also valid for $\boldsymbol{\Pi}^1_m$). First assume that $p_\alpha\Vdash\exists_{z\in\omega^\omega}\psi(\dot{f}_\alpha,z)$ for $\Dwf$-many $\alpha$ and, for those $\alpha$, choose a $\Por$-name $\dot{g}_\alpha$ such that $p_\alpha\Vdash\psi(\dot{f}_\alpha,\dot{g}_\alpha)$. By induction hypothesis, $\bar{p}\Vdash\psi(\dot{f},\dot{g})$ where $\dot{g}=\langle\dot{g}_\alpha\rangle_{\alpha<\kappa}/\Dwf$. The converse is also easy.
\end{proof}

\begin{corollary}[Shelah {\cite{Sh-TempIt}}, see also {\cite[Lemma 0.3]{Br-TempIt}}]\label{UltrapowDestrMAD}
   Let $\dot{\Awf}$ be a $\Por$-name of an almost disjoint (a.d.) family such that $\Vdash_\Por|\Awf|\geq\kappa$. Then, $\Vdash_{\Por^\kappa/\Dwf}\Awf\textrm{ is not maximal}$.
\end{corollary}
\begin{proof}
   Let $r\in\Por$ and $\lambda\geq\kappa$ be a cardinal such that $r\Vdash_\Por\dot{\Awf}=\{\dot{A}_\xi\ /\ \xi<\lambda\}$. Put $\dot{A}=\langle\dot{A}_\alpha\rangle_{\alpha<\kappa}/\Dwf$ (this can be defined in a similar way by associating the characteristic function to each set), and show that it is a $\Por^\kappa/\Dwf$-name of an infinite subset of $\omega$ and $i(r)\Vdash\forall_{\xi<\lambda}(|\dot{A}_\xi\cap\dot{A}|<\aleph_0)$. But this is straightforward from Theorem \ref{ProjStatandUltrapow}.
\end{proof}

{\small

\bibliographystyle{spmpsci}      


}
\end{document}